\documentclass[11pt]{article}

\usepackage[margin=1.1in]{geometry}
\usepackage{amsmath,amssymb,amsfonts,amsthm,mathrsfs}
\usepackage{mathtools}
\usepackage{enumitem}
\usepackage{booktabs}
\usepackage{array}
\usepackage{hyperref}
\usepackage{multirow}
\usepackage{longtable}

\hypersetup{colorlinks=true,linkcolor=blue,citecolor=blue,urlcolor=blue}

\newtheorem{thm}{Theorem}[section]
\newtheorem{prop}{Proposition}[section]
\newtheorem{lem}{Lemma}[section]

\newtheorem{rmk}{Remark}[section]

\DeclareMathOperator{\Gr}{Gr}
\DeclareMathOperator{\IG}{IG}
\DeclareMathOperator{\Aut}{Aut}

\DeclareMathOperator{\Graph}{Graph}

\newcommand{\F}{\mathbb F}
\newcommand{\R}{\mathbb R}
\newcommand{\C}{\mathbb C}
\newcommand{\HH}{\mathbb H}
\newcommand{\cX}{\mathcal X}
\newcommand{\cT}{\mathcal T}

\newcommand{\olG}{\overline G}
\newcommand{\olGo}{\olG^{\circ}}
\newcommand{\eps}{\varepsilon}
\newcommand{\GL}{\mathrm{GL}}
\newcommand{\OO}{\mathrm{O}}
\newcommand{\UU}{\mathrm{U}}
\newcommand{\SpG}{\mathrm{Sp}}
\newcommand{\Ostar}{\mathrm{O}^{*}}

\newcolumntype{M}[1]{>{\raggedright\arraybackslash}m{#1}}
\newcolumntype{L}[1]{>{\raggedright\arraybackslash}p{#1}}
\newcolumntype{C}[1]{>{\centering\arraybackslash}p{#1}}

\title{Grassmannian Realizations of Classical Pseudo-Riemannian Symmetric Spaces}
\author{Yunxia Chen \& Naichung Conan Leung}
\date{}

\begin{document}
	\maketitle

\begin{abstract}
		We establish a unified fixed-point Grassmannian framework for the classical pseudo-Riemannian symmetric spaces from Berger’s classification.
		For a classical group $G=\Aut(V,b)$, we identify $G$ with the open graph-like locus inside an ordinary or maximal-isotropic Grassmannian $\overline{G}$.
		Natural involutions of $\overline{G}$ are constructed via classical forms, compatible scalar structures, and subspace configurations.
		We explicitly describe the full fixed loci of these involutions in terms of classical Grassmannians.
		The corresponding graph-like fixed loci (or their signature strata) are identified, respectively, with the moduli spaces of classical forms, compatible scalar structures, and subspace configurations; these moduli spaces coincide with classical pseudo-Riemannian symmetric spaces.
		In summary, the $54$ classical series are organized into four families: $10$ group-type models, $7$ classical-form Grassmannians, $20$ scalar-structure Grassmannians, and $17$ subspace-configuration Grassmannians.
\end{abstract}

\section{Introduction}
Pseudo-Riemannian symmetric spaces generalize Riemannian symmetric spaces. Every such semisimple space can be represented as a homogeneous space $G/H$ determined by an involutive automorphism $\rho$ of a Lie group $G$, where $H$ is an open subgroup of the fixed-point subgroup $G^\rho$ \cite{Huang, Loos, R}. Berger established the systematic classification of semisimple pseudo-Riemannian symmetric spaces \cite{B}. Under Neretin’s convention, this classification includes 54 classical series and 131 exceptional series, up to finite coverings and connected components \cite{N}.

The present paper concerns the $54$ classical series. A homogeneous space $G/H$ is called classical here when $G$ is a classical Lie group. According to the standard classification \cite{Har,K}, every such classical group can be realized as $G=\Aut(V,b)$, where $V$ is a linear space over $\F \in \{\R, \C, \HH \}$, i.e. $V \cong \F^n$, and  $b=0$ or $b$ is a nondegenerate classical form on $V$
(a symmetric or skew-symmetric form over $\R$ or $\C$, or a Hermitian or skew-Hermitian form over $\C$ or $\HH$).
The explicit correspondence between underlying fields, classical forms, and the associated classical  groups is summarized in the following table\footnote{See Section 1.1 in \cite{Har} for these classical forms and their symmetry groups. Over $\mathbb{C}$, skew-Hermitian forms have associated groups $\UU(p,q)$, same as Hermitian forms. When compatibility enforces a split complex skew-Hermitian form in later constructions, we denote its group by $\UU(m,m)$.}:
\vspace{-8pt}
\begin{table}[htbp]
	\centering
	\caption{Classical groups associated with classical forms}
	\label{tab:gp}
	\begin{tabular}{|c|c|c|c|c|c|}
		\hline
		$\mathbb{F}$ $\backslash$ $b$    & $b=0$ & \small{symmetric}  & \small{skew-symmetric}  & \small{Hermitian} & \small{skew-Hermitian}  \\  \hline
		$\mathbb{R}$  & $\GL(n, \mathbb{R})$ & $\OO(p,q)$ & $\SpG(2m,\mathbb{R})$ &  &  \\  \hline
		$\mathbb{C}$ & $\GL(n, \mathbb{C})$ &  $\OO(n,\mathbb{C})$ & $\SpG(2m,\mathbb{C})$ & $\UU(p,q)$ & $\UU(p,q)$ \\ \hline
		$\mathbb{H}$ & $\GL(n, \mathbb{H})$  &  &  & $\SpG(p,q)$ & $\Ostar(2n)$ \\ \hline
	\end{tabular}
\end{table}

A pseudo-Riemannian symmetric space $G/H$ is Riemannian if and only if the isotropy subgroup $H$ is compact. In \cite {HL}, the authors constructed a unified geometric realization of all compact Riemannian symmetric spaces in terms of Grassmannians. In the classical setting, there are three ordinary Grassmannian series and seven maximal-isotropic Grassmannian series;
the corresponding space-like Grassmannians realize the noncompact duals of these spaces \cite{CHL}.

This raises a natural question: can the classical pseudo-Riemannian symmetric spaces themselves be realized geometrically in Grassmannians? Neretin gave  algebraic realizations of all classical series in terms of complementary subspaces, together with open Grassmannian embeddings \cite{N}.
In this paper, we establish a unified fixed-point framework for all $54$ classical series: each symmetric space is realized as a distinguished graph-like locus, or signature stratum, inside the fixed-point locus of a natural involution of a Grassmannian compactification.

In \cite{He}, H.Y. He constructed Grassmannian compactifications of classical groups. For $G=\Aut(V)$, the compactification is $\olG=\Gr_{\mathbb{F}}(n,V\oplus V)$, while for $G=\Aut(V,b)$ with $b\neq0$, $\olG=\IG(V\oplus V,b\oplus(-b))$. In either case, the open graph-like locus $\olG^\circ=\{W\in\olG: W\cap(V\oplus0)=W\cap(0\oplus V)=0\}$ is canonically identified with $G$ by $g\longmapsto\operatorname{Graph}(g)$.

The involutions on $\olG$ considered below arise from three families of auxiliary data: 
(i) a nondegenerate classical form on $V$; 
(ii) a compatible scalar structure on $(V,b)$, namely a compatible complex structure on $(\mathbb R^{2n},b)$, a compatible quaternionic structure on $(\mathbb C^{2n},b)$, a compatible real form structure on $(\mathbb C^{n},b)$, or a compatible complex form structure on $(\mathbb H^{n},b)$; 
(iii) a configuration of subspaces in $V$, such as an isotropic pair of subspaces, a nondegenerate subspace, or a complementary pair of subspaces.

Apart from the group-type case, our construction follows a common pattern:
	\begin{enumerate}
		\item We start with a classical group $G=\Aut(V,b)$, its Grassmannian compactification $\olG$, and the open graph-like locus $\olGo$.
		\item We construct a natural involution $R$ of $\olG$ from one of the above families of auxiliary data. We compute the fixed-point locus $\olG^R$, which admits an explicit description in terms of classical Grassmannians, sometimes as a product or disjoint union of these Grassmannians.
		\item We intersect $\olG^R$ with $\olGo$ and identify the result with a moduli space of the same family of auxiliary data, (i), (ii) or (iii), used to define the involution. We prove that this graph-like locus (or its signature stratum, where applicable) of $\olG^R$ is diffeomorphic to the pseudo-Riemannian symmetric space $G/H$.
\end{enumerate}

Based on the above constructions, we organize the classical series into four Grassmannian families, as described below.

\begin{thm}
	Up to finite coverings  and connected components, the 54 classical series of pseudo-Riemannian symmetric spaces admit fixed-point Grassmannian models, organized into four families:
	\begin{enumerate}
		\item Group-type models: 10 series.
		\item Classical-form Grassmannians: 7 series, parametrizing nondegenerate classical forms of prescribed algebraic type, with signature strata where applicable.
		\item Scalar-structure Grassmannians: 20 series, comprising the complex-structure, quaternionic-structure, real-form, and  complex-form families.
		\item Subspace-configuration Grassmannians: 17 series, comprising isotropic-pair, nondegenerate-subspace, and  complementary-pair families.
	\end{enumerate}
\end{thm}

The following table \ref{tab:54-classical-series} lists all $54$ series explicitly.
{\footnotesize
	\renewcommand{\arraystretch}{1.6}
	\setlength{\abovedisplayskip}{0pt}
	\setlength{\belowdisplayskip}{0pt}
	\setlength{\abovedisplayshortskip}{0pt}
	\setlength{\belowdisplayshortskip}{0pt}
	\begin{longtable}{| M{0.04\textwidth} | M{0.18\textwidth} | C{0.03\textwidth} | M{0.70\textwidth} |}
		\caption{The $54$ classical series in the four Grassmannian families}
		\label{tab:54-classical-series}
		\\
		\hline
		\textbf{Type} & \textbf{Grassmannian family} & \textbf{No.} & \textbf{Pseudo-Riemannian symmetric spaces}
		\\
		\hline
		\endfirsthead
		\multicolumn{4}{|c|}{\tablename\ \thetable\ --- \textit{continued}} \\
		\hline
		\textbf{Type} & \textbf{Grassmannian family} & \textbf{No.} & \textbf{Pseudo-Riemannian symmetric spaces}
		\\
		\hline
		\endhead
		\multicolumn{4}{|r|}{\textit{continued}} \\
		\endfoot
		\hline
		\endlastfoot
		
		\textbf{I} & Group-type models & 10
		& $\GL(n,\mathbb F),\; \OO(p,q),\; \UU(p,q),\; \SpG(p,q),\; \SpG(2m,\mathbb R),\; \SpG(2m,\mathbb C),\; \OO(n,\mathbb C),\; \Ostar(2n).$
		\\
		\hline
		
		\textbf{II} & Classical-form Grassmannians & 7
		& $\begin{aligned}[t]
			&\dfrac{\GL(n,\mathbb R)}{\OO(p,q)},\:
			\dfrac{\GL(2m,\mathbb R)}{\SpG(2m,\mathbb R)},\:
			\dfrac{\GL(n,\mathbb C)}{\OO(n,\mathbb C)},\:
			\dfrac{\GL(2m,\mathbb C)}{\SpG(2m,\mathbb C)},\:
			\dfrac{\GL(n,\mathbb C)}{\UU(p,q)},\:
			\dfrac{\GL(n,\mathbb H)}{\SpG(p,q)},\:
			\dfrac{\GL(n,\mathbb H)}{\Ostar(2n)}.
		\end{aligned}$
		\\
		\hline
		
		\multirow{4}{*}{\textbf{III}}
		& Complex-structure Grassmannians & 5
		& $\dfrac{\GL(2n,\mathbb R)}{\GL(n,\mathbb C)},\:
		\dfrac{\OO(2p,2q)}{\UU(p,q)},\:
		\dfrac{\SpG(2n,\mathbb R)}{\UU(p,q)},\:
		\dfrac{\OO(n,n)}{\OO(n,\mathbb C)},\:
		\dfrac{\SpG(4m,\mathbb R)}{\SpG(2m,\mathbb C)}.$
		\\
		\cline{2-4}
		& Quaternionic-structure Grassmannians & 5
		& $\dfrac{\GL(2n,\mathbb C)}{\GL(n,\mathbb H)},\:
		\dfrac{\OO(2n,\mathbb C)}{\Ostar(2n)},\:
		\dfrac{\SpG(2n,\mathbb C)}{\SpG(p,q)},\:
		\dfrac{\UU(2p,2q)}{\SpG(p,q)},\:
		\dfrac{\UU(n,n)}{\Ostar(2n)}.$
		\\
		\cline{2-4}
		& Real-form Grassmannians & 5
		& $\dfrac{\GL(n,\mathbb C)}{\GL(n,\mathbb R)},\:
		\dfrac{\OO(n,\mathbb C)}{\OO(p,q)},\:
		\dfrac{\SpG(2m,\mathbb C)}{\SpG(2m,\mathbb R)},\:
		\dfrac{\UU(p,q)}{\OO(p,q)},\:
		\dfrac{\UU(m,m)}{\SpG(2m,\mathbb R)}.$
		\\
		\cline{2-4}
		& Complex-form Grassmannians & 5
		& $\dfrac{\GL(n,\mathbb H)}{\GL(n,\mathbb C)},\:
		\dfrac{\SpG(p,q)}{\UU(p,q)},\:
		\dfrac{\Ostar(2n)}{\UU(p,q)},\:
		\dfrac{\SpG(m,m)}{\SpG(2m,\mathbb C)},\:
		\dfrac{\Ostar(2n)}{\OO(n,\mathbb C)}.$
		\\
		\hline
		
		\multirow{3}{*}{\textbf{IV}}
		& Isotropic-pair Grassmannians & 7
		& $\dfrac{\OO(n,n)}{\GL(n,\mathbb R)},\:
		\dfrac{\UU(n,n)}{\GL(n,\mathbb C)},\:
		\dfrac{\SpG(n,n)}{\GL(n,\mathbb H)},\:
		\dfrac{\SpG(2n,\mathbb R)}{\GL(n,\mathbb R)},\:
		\dfrac{\SpG(2n,\mathbb C)}{\GL(n,\mathbb C)},\:
		\dfrac{\OO(2n,\mathbb C)}{\GL(n,\mathbb C)},\:
		\dfrac{\Ostar(4n)}{\GL(n,\mathbb H)}.$
		\\
		\cline{2-4}
		& Nondegenerate-subspace Grassmannians & 7
		& $\begin{aligned}[t]
			&\dfrac{\OO(r,s)}{\OO(p,q)\times \OO(r-p,s-q)},\:
			\dfrac{\UU(r,s)}{\UU(p,q)\times \UU(r-p,s-q)},\:
			\dfrac{\SpG(r,s)}{\SpG(p,q)\times \SpG(r-p,s-q)},\\
			&\dfrac{\SpG(2n,\mathbb R)}{\SpG(2k,\mathbb R)\times \SpG(2n-2k,\mathbb R)},\:
			\dfrac{\SpG(2n,\mathbb C)}{\SpG(2k,\mathbb C)\times \SpG(2n-2k,\mathbb C)},\\
			&\dfrac{\OO(n,\mathbb C)}{\OO(k,\mathbb C)\times \OO(n-k,\mathbb C)},\:
			\dfrac{\Ostar(2n)}{\Ostar(2k)\times \Ostar(2n-2k)}.
		\end{aligned}$
		\\
		\cline{2-4}
		& Complementary-pair Grassmannians & 3
		& $\dfrac{\GL(n,\mathbb R)}{\GL(k,\mathbb R)\times \GL(n-k,\mathbb R)},\:
		\dfrac{\GL(n,\mathbb C)}{\GL(k,\mathbb C)\times \GL(n-k,\mathbb C)},\:
		\dfrac{\GL(n,\mathbb H)}{\GL(k,\mathbb H)\times \GL(n-k,\mathbb H)}.$
		\\
		\hline
		\multicolumn{2}{|r|}{\textbf{Total}} & \textbf{54} & $10+7+20+17=54$
		\\
		\hline
	\end{longtable}
}

\begin{rmk}
	In comparison with Neretin's classification, our Type IV.1 (isotropic-pair Grassmannians) corresponds to List 2, our Type IV.2 (nondegenerate-subspace Grassmannians) to List 4, and our Type IV.3 (complementary-pair Grassmannians) to List 5 in \cite{N}. In addition, List 3 consists of three classical automorphism groups $\Aut(V)$ together with four scalar-structure Grassmannians associated with $\Aut(V)$. Meanwhile, List 1 contains seven classical groups $\Aut(V,b)$, sixteen scalar-structure Grassmannians associated with $\Aut(V,b)$, and the complete family of seven classical-form Grassmannians.
\end{rmk}

The paper is organized as follows.
Section \ref{sec:group} is devoted to Group-type models.
Section \ref{sec:b-related} investigates Classical-form Grassmannians.
Section \ref{sec:F-related} addresses Scalar-structure Grassmannians.
Section \ref{sec:subspace} studies Subspace-configuration Grassmannians.
Finally, Section \ref{sec:rmk} contains the concluding remarks.

\bigskip

$\mathbf{Acknowledgements.}$  The first author is supported by the National Natural
Science Foundation of China (No. 12471045). The second author is
supported by  research grants from the Research Grants Council of the Hong
Kong Special Administrative Region, China (No. 14302225, No. 14302224, No. 14305923).
The authors confirm that no AI tools were utilized in any stage of this research or manuscript preparation.

\bigskip
\section{Type I: Group-type models}\label{sec:group}
We begin by recalling the classical Grassmannians that serve as the ambient compact manifolds throughout the paper.
Let $\mathbb{F}=\mathbb{R},\mathbb{C},\mathbb{H}$, and let $V\cong \mathbb{F}^n$.
We write $\Gr_{\mathbb{F}}(k,V)$ for the Grassmannian of $k$-dimensional $\mathbb{F}$-subspaces of $V$.
The ordinary real, complex, and quaternionic Grassmannians have the standard homogeneous-space realizations
\[
\Gr_{\mathbb{R}}(k,\mathbb{R}^n)\cong \frac{\OO(n)}{\OO(k) \times \OO(n-k)},   \Gr_{\mathbb{C}}(k,\mathbb{C}^n)\cong \frac{\UU(n)}{\UU(k) \times \UU(n-k)},   
\Gr_{\mathbb{H}}(k,\mathbb{H}^n)\cong \frac{\SpG(n)}{\SpG(k) \times \SpG(n-k)}.
\]

Now let $V\cong\mathbb{F}^{2n}$ be equipped with a nondegenerate split classical form $b$\footnote{For symmetric and Hermitian forms, split means maximal Witt index; in the signature cases this is equivalent to equal positive and negative indices.}.
We denote by
\[
\IG(V,b)
=\left\{W\in \Gr_{\mathbb{F}}(n,V): b|_{W}=0\right\}
\]
the Grassmannian of maximal $b$-isotropic subspaces of $V$.
The seven possibilities occurring for split classical forms are summarized in Table \ref{tab:IG} below.
\vspace{-8pt}
\begin{table}[htbp]
	\centering
	\caption{Maximal-isotropic Grassmannians \( IG(V,b) \) for split classical forms \( b \)}
	\label{tab:IG}
	\renewcommand{\arraystretch}{1.15}\setlength{\tabcolsep}{1pt}
	\begin{tabular}{|c|c|c|c|c|c|c|c|}
		\hline
		\(V\) & \(\mathbb{R}^{2n}\) & \(\mathbb{R}^{2n}\) & \(\mathbb{C}^{2n}\) 
		& \(\mathbb{C}^{2n}\) 	& \(\mathbb{C}^{2n}\) & \(\mathbb{H}^{2n}\) & \(\mathbb{H}^{2n}\) \\
		\hline
		\(b\) & \small{\((n,n)\)-symm.} & \small{skew-symm.} 
		& \small{symm.} & \small{skew-symm.} 	& \small{\((n,n)\)-Herm.}
		& \small{\((n,n)\)-Herm.} & \small{skew-Herm.} \\
		\hline
		\small{\(IG(V,b)\)} & \(\OO(n)\) & \(\frac{\UU(n)}{\OO(n)}\) & \(\frac{\OO(2n)}{\UU(n)}\) 
		& \(\frac{\SpG(n)}{\UU(n)}\)  & \(\UU(n)\) & \(\SpG(n)\) & \(\frac{\UU(2n)}{\SpG(n)}\) \\
		\hline
	\end{tabular}
\end{table}

We now turn to the group-type symmetric spaces.
Let $(V\cong \mathbb{F}^n,b)$ be a classical space, where either $b=0$ or $b$ is a nondegenerate classical form, and set
\[
G=\Aut(V,b).
\]
When $b=0$, this means $G=\GL(n, \mathbb{F})$.
Put
\[
E=V\oplus V,\qquad B=b\oplus(-b).
\]
By \cite{He}, the Grassmannian compactification of $G$ is
\[
\olG
=
\begin{cases}
	\Gr_{\mathbb{F}}(n,E), & b=0,\\[1mm]
	\IG(E,B), & b\neq 0.
\end{cases}
\]
We denote by
\[
\olGo
=\left\{
W\in \olG:\;
W\cap(V\oplus 0)=0,\;
W\cap(0\oplus V)=0
\right\}
\]
its graph-like locus.

The following proposition gives the canonical realization of $G$ inside $\olG$ \cite{CL}.

\begin{prop}
	The map
	\[
	\Gamma\colon G\to \olGo,\quad g\mapsto \Graph(g),
	\]
	is a diffeomorphism. Moreover, $\Gamma$ is $G\times G$-equivariant, where $G\times G$ acts on $E=V\oplus V$ by
	\[
	(g_1,g_2)\cdot(x,y)=(g_1 x,g_2 y),
	\]
	and on $G$ by
	\[
	(g_1,g_2)\cdot g = g_2 g g_1^{-1}.
	\]
\end{prop}
\begin{proof}
	For any $g\in G$,
	\[
	B\big((u,gu),(v,gv)\big)=b(u,v)-b(gu,gv)=0,
	\]
	so $\Graph(g)$ is maximal $B$-isotropic. Invertibility of $g$ gives the transversality conditions, thus $\Graph(g)\in\olGo$.

		Conversely, take $W\in\olGo$. Since $W\cap(0\oplus V)=0$ and $\dim_{\mathbb{F}}W=\dim_{\mathbb{F}}V$, $\operatorname{pr}_1|_W\colon W\to V$ is an isomorphism. Define
	\[
	T=\operatorname{pr}_2\circ\bigl(\operatorname{pr}_1|_W\bigr)^{-1}\colon V\to V,
	\]
	so $W=\Graph(T)=\{(u,Tu)\mid u\in V\}$.
	
	If $b=0$, then $T\in \GL(n, \mathbb{F})=G$.
	If $b\neq 0$, $W\in\IG(E,B)$ implies
	\[
	0=B\big((u,Tu),(v,Tv)\big)=b(u,v)-b(Tu,Tv),
	\]
	so $b(Tu,Tv)=b(u,v)$ for all $u,v\in V$, hence $T\in \Aut(V,b)=G$. The map $\Gamma$ is surjective.

	The inverse map $W\mapsto \operatorname{pr}_2\circ(\operatorname{pr}_1|_W)^{-1}$ is smooth on $\olGo$, hence $\Gamma$ is a diffeomorphism.
	Equivariance follows from
	\[
	(g_1,g_2)\cdot \Graph(g)
	=\{(g_1 u,g_2 g u)\mid u\in V\}
	=\Graph(g_2 g g_1^{-1}).
	\qedhere
	\]
\end{proof}

We now describe the realization of $G$ as a symmetric space.
\begin{prop}
	There are $G\times G$-equivariant diffeomorphisms
	\[
	\olGo\cong G \cong \frac{G\times G}{\Delta G},
	\]
	where $\Delta G=\{(g,g)\mid g\in G\}$.
	Thus $\olGo$ is a group-type pseudo-Riemannian symmetric space.
\end{prop}
\begin{proof}
	The $G\times G$-action $(g_1,g_2)\cdot g = g_2 g g_1^{-1}$ on $G$ is transitive, with stabilizer of the identity equal to $\Delta G$, so $G\cong (G\times G)/\Delta G$.
	Define the involution $\tau\colon G\times G\to G\times G$ by $\tau(g_1,g_2)=(g_2,g_1)$. Its fixed-point subgroup is exactly $\Delta G$, so $(G\times G,\Delta G)$ forms a symmetric pair, and the quotient $\frac{G\times G}{\Delta G}$ is a pseudo-Riemannian symmetric space.
	\qedhere
\end{proof}

Under our conventions, this gives the ten classical group-type series:
\[
\GL(n,\R),\ \GL(n,\C),\ \GL(n,\HH),\ \OO(p,q),\ \UU(p,q),\ \SpG(p,q),\ \SpG(2n,\R),\ \SpG(2n,\C),\ \OO(n,\C),\ \Ostar(2n),
\]
each regarded as the group-type symmetric space $(G\times G)/\Delta G$.

\bigskip
\section{Type II: Classical-form Grassmannians}\label{sec:b-related}
Let $V\cong\F^n$, where $\F=\R,\C,\HH$, and set $E=V\oplus V$, $G=\GL(n, \mathbb{F})$, and $\olG=\Gr_{\F}(n,E)$. We retain $\olGo$ for the  graph-like locus of $\olG$.

Let \(b\) be a nondegenerate classical form on \(V\).
Define an associated form $\beta$ on $E$ by
\[
\beta\colon E\times E\to\mathbb{F},\qquad
\beta((x,y),(x',y'))
= b(x,y')-b(y,x').
\]
The algebraic type of \(\beta\) is determined by that of \(b\):
\begin{enumerate}[label=\textup{(\arabic*)}]
	\item   $V \cong \R^n$, $b$  symmetric $\implies$ $\beta$   skew-symmetric;
	\item  $V \cong \R^{2m}$, $b$ skew-symmetric  $\implies$  $\beta$  symmetric with signature $(2m,2m)$;
	\item  $V \cong \C^n$, $b$ symmetric $\implies$ $\beta$   skew-symmetric;
	\item  $V \cong \C^{2m}$, $b$ skew-symmetric  $\implies$  $\beta$  symmetric;
	\item  $V \cong \C^n$, $b$ Hermitian $\implies$ $\beta$   skew-Hermitian with signature $(n,n)$;
	\item  $V \cong \HH^n$, $b$ Hermitian $\implies$ $\beta$   skew-Hermitian;
	\item  $V \cong \HH^n$, $b$ skew-Hermitian $\implies$ $\beta$  Hermitian with signature $(n,n)$.
\end{enumerate}
Note that all the above $\beta$ are split forms. Accordingly, for each case we obtain the associated Grassmannian
$\IG(E,\beta)$, which consists of all maximal  $\beta$-isotropic subspaces of $E$.

Since \(\beta\) is nondegenerate, orthogonal complementation defines a smooth involution
\[
R_\beta\colon \olG\longrightarrow\olG,\qquad
W\longmapsto W^{\perp_\beta}.
\]
Its fixed points are exactly the maximal $\beta$-isotropic subspaces.

\begin{prop}
	The fixed-point locus of $R_\beta$ is
	\[
	\olG^{R_\beta}=\IG(E,\beta).
	\]
\end{prop}

We now restrict to the graph-like locus and set
\[
\mathcal{X}_b:=\olG^{R_\beta}\cap\olGo
=\{W\in\olGo\,|\, R_\beta(W)=W\}.
\]

\begin{prop}
	There exists a canonical smooth diffeomorphism
	\[
	\mathcal{X}_b \cong \mathcal{T}_b,
	\]
	where $\mathcal{T}_b$ is the space of nondegenerate classical forms on $V$ of the same algebraic type as $b$.
\end{prop}
\begin{proof}
	Any $W\in\mathcal{X}_b$ is the graph of a unique invertible linear operator $T\in G=\GL(n,\mathbb{F})$, i.e.,
	\[
	W=\Graph(T)=\{(u,Tu)\,|\,u\in V\}.
	\]
	The fixed-point condition $R_\beta(W)=W$ is equivalent to the self-adjointness relation
	\[
	b(u,Tv)=b(Tu,v),\quad \forall\,u,v\in V.
	\]
	Thus
	\[
	\mathcal{X}_b\cong\mathcal{T}_b':=\{T\in G\,|\,b(u,Tv)=b(Tu,v),\,\forall u,v\in V\}.
	\]
	
	Define the smooth map
	\[
	f\colon\mathcal{T}_b'\to\mathcal{T}_b,\qquad T\mapsto a_T:=b(\,\cdot\,,T\,\cdot\,).
	\]
	Since $b$ is nondegenerate and $T$ is invertible, $a_T$ is a nondegenerate form on $V$. The operator relation $b(u,Tv)=b(Tu,v)$ guarantees that $a_T$ shares the same classical type as $b$. Conversely, for $a\in\mathcal{T}_b$,
	the common algebraic type of $a$ and $b$ gives a unique invertible $T$ satisfying $b(u,Tv)=b(Tu,v)$ and $a(u,v)=b(u,Tv)$. Thus $f$ is a diffeomorphism.
\end{proof}

The group \(G=\GL(n, \mathbb{F})\) acts on \(\mathcal{T}_b\) by congruence,
\[
(g\cdot a)(u,v)
= a(g^{-1}u,g^{-1}v).
\]
Its orbits are exactly the isometry classes of classical forms with the same algebraic type as $b$.

If \(b\) is real skew-symmetric, complex symmetric, complex skew-symmetric, or quaternionic skew-Hermitian, there is a single isometry class in fixed dimension, and \(G\) acts transitively on \(\mathcal{T}_b\).
If \(b\) is real symmetric, complex Hermitian, or quaternionic Hermitian, then
\[
\mathcal{T}_b
=
\bigsqcup_{p+q=n}\mathcal{T}_{p,q},\qquad
\mathcal{T}_{p,q}
=
\{a\in\mathcal{T}_b:\operatorname{sign}(a)=(p,q)\},
\]
and \(G\) acts transitively on each \(\mathcal{T}_{p,q}\).

For \(a\in\mathcal{T}_b\), let
\[
G_a
=
\Aut(V,a)
=
\{g\in G:a(gu,gv)=a(u,v)\}.
\]
Then the orbit of \(a\) is naturally
\[
G\cdot a\simeq G/G_a.
\]

\begin{prop}
	For every \(a\in\mathcal{T}_b\), the homogeneous space \(G/G_a\) is a pseudo-Riemannian symmetric space.
\end{prop}
\begin{proof}
	Let \(g^{\dagger}_a\) denote the adjoint of \(g\) with respect to \(a\), defined by
	\[
	a(gu,v)=a(u,g^{\dagger}_a v).
	\]
	Then
	\[
	\rho_a\colon G\longrightarrow G,\qquad
	\rho_a(g)=(g^{\dagger}_a)^{-1},
	\]
	is an involutive automorphism. Its fixed-point subgroup is
	\[
	G^{\rho_a}
	= \{g\in G:g^{\dagger}_a=g^{-1}\}
	= G_a.
	\]
	Thus $(G,G_a)$ is a symmetric pair, so $G/G_a$ is a pseudo-Riemannian symmetric space.
\end{proof}

Combining the preceding results gives the following seven classical series.

\begin{thm}[Classical-form Grassmannians]
	The graph-like fixed locus \(\mathcal{X}_b\simeq\mathcal{T}_b\) is either a single symmetric-space orbit or, in the signature-dependent cases, a disjoint union of symmetric-space orbits. We call these orbits \textbf{classical-form Grassmannians}, as they parameterize  nondegenerate classical forms of a fixed algebraic type on $V$. Explicitly, this construction produces the following seven series:
	\[
	\begin{gathered}
		\frac{\GL(n,\mathbb{R})}{\OO(p,q)},\;
		\frac{\GL(2m,\mathbb{R})}{\SpG(2m,\mathbb{R})},\;
		\frac{\GL(n,\mathbb{C})}{\OO(n,\mathbb{C})},\;
		\frac{\GL(2m,\mathbb{C})}{\SpG(2m,\mathbb{C})},\;
		\frac{\GL(n,\mathbb{C})}{\UU(p,q)},\;
		\frac{\GL(n,\mathbb{H})}{\SpG(p,q)},\;
		\frac{\GL(n,\mathbb{H})}{\Ostar(2n)}.
	\end{gathered}
	\]
\end{thm}

\bigskip
\section{Type III: Scalar-structure Grassmannians}\label{sec:F-related}
The family of scalar-structure Grassmannians is obtained from the constructions
$\R^{2n} \to \C^n$ and $\C^{2n} \to \HH^n$, together with 
$\C^n \to \R^n$ and $\HH^n \to \C^n$. In this way we obtain four types of spaces,
namely \emph{complex-structure Grassmannians}, \emph{quaternionic-structure
	Grassmannians}, \emph{real-form Grassmannians}, and \emph{complex-form
	Grassmannians}. Since these four constructions are entirely parallel, we first
present the common framework below and then discuss each type in a separate
subsection.

Start with  a classical space $(V \cong \F^{n}, b)$, set $G=\Aut(V,b)$, $E=V\oplus V$, $B=b\oplus(-b)$.
Denote by $\olG$ the Grassmannian compactification of $G$ and by
$\olGo$ its graph-like locus, as before.

Suppose that $V$ is endowed with a map $\sigma: V\to V$ satisfying the following properties:

\boxed{ \begin{array}{c|c|c|c|c}
		
		\text{type of Grassmannian} & \F  & \text{linearity} & \sigma^2 &  \text{$b$-compatibility} \\  \hline
		
		\text{complex-structure Grassmannian}	&  \R     &	\text{$\R$-linear}		&	-I	&	b(\sigma u,\sigma v)=\pm b(u,v)	
		\\ [1 mm] \hline
		
		\text{quaternionic-structure Grassmannian}  & \C	 & \text{conjugate-linear} 	& -I	& 	 \raisebox{-3pt}{$b(\sigma u, \sigma v)=\overline{b(u,v)}$}
		\\ [1 mm] \hline
		
		\text{real-form Grassmannian}	&	\C	 	 & \text{conjugate-linear}	 	& +I	& 	\raisebox{-3pt}{$b(\sigma u, \sigma v)=\overline{b(u,v)}$}
		\\ [1 mm] \hline
		
		\text{complex-form Grassmannian}  &\HH    & \text{$\HH$-linear}	&	-I  &	b(\sigma u,\sigma v)=\pm b(u,v)	
		
\end{array}}

\bigskip
Define the map $R_{\sigma}:E\to E$ by
\[
R_{\sigma}(x,y)=(\sigma y,\sigma x),\qquad \forall x,y\in V.
\]
A direct computation then yields $R_{\sigma}^{2}=I$ (resp.\ $R_{\sigma}^{2}=-I$)
according as $\sigma^{2}=I$ (resp.\ $\sigma^{2}=-I$).

\begin{lem}\label{R}
	The map $R_\sigma$ induces an involution on $\olG$.
\end{lem}
\begin{proof}
	When $b=0$, $R_\sigma^2=\pm I_E$, and $\pm I_E$ acts trivially on the Grassmannian; hence $R_\sigma$ induces an involution on $\olG=\Gr_{\F}(n,E)$.
	
	When $b\ne 0$, for any $z=(x,y), w=(x',y')\in E$, we have
	
	(1) for $\F =\R, \HH$, $b(\sigma u,\sigma v)=\eps b(u,v)$ with $\eps=\pm 1$,
	\[
	\begin{aligned}
		B(R_\sigma(z),R_\sigma(w))
		&=B((\sigma y,\sigma x),(\sigma y',\sigma x')) \\
		&=b(\sigma y,\sigma y')-b(\sigma x,\sigma x') \\
		&=\eps b(y,y')-\eps b(x,x') \\
		&=- \eps B(z,w);
	\end{aligned}
	\]
	
	(2) for $\F =\C$, $b(\sigma u,\sigma v)=\overline{b(u,v)}$,
	\[
	\begin{aligned}
		B(R_\sigma (z),R_\sigma(w))
		&=B((\sigma y,\sigma x),(\sigma y',\sigma x')) \\
		&=b(\sigma y,\sigma y')-b(\sigma x,\sigma x') \\
		&=\overline{b(y,y')}
		-\overline{b(x,x')}\\
		&=-\overline{B(z,w)}.
	\end{aligned}
	\]
	
	It follows that $R_\sigma$ sends maximal $B$-isotropic subspaces to maximal $B$-isotropic subspaces. Hence $R_{\sigma}$  induces an involution on $\olG=\IG_{\F}(E,\, B)$.
\end{proof}

The description of the full fixed-point locus $\olG^{R_\sigma}$  depends on the scalar structure and will therefore be carried out separately in each of the four subsections below. We first record two results that are common to all four constructions: the identification of the graph-like locus with a moduli space, and the corresponding symmetric-pair description.

We now consider the  graph-like locus of $\olG^{R_\sigma}$, i.e.
\[
\cX:=\olG^{R_\sigma}\cap \olGo=\{ W \in \olGo: R_\sigma(W)=W\}.
\]

\begin{prop}\label{X-T}
	There is a natural diffeomorphism
	\[
	\cX \cong \cT :=\left\{
	A:V \to V | A \text{ has the same linearity type, square, and $b$-compatibility as } \sigma
	\right\}.
	\]
\end{prop}
\begin{proof}
	For $W\in\cX$, write $W=\Graph(S)$ with $S\in G$. Then
	\[
	R_\sigma(W)=W \iff S\sigma S=\sigma,
	\]
	so
	\[
	\cX\cong\cT':=\{S\in G:S\sigma S=\sigma\}.
	\]
	The map $S\mapsto A=S\sigma$ sends $\cT'$ to $\cT$. Conversely, if $A\in\cT$, then $S=A\sigma^{-1}$ lies in $G$ because $A$ and $\sigma$ have the same $b$-compatibility, and
	\[
	S\sigma S=A^2\sigma^{-1}=\sigma.
	\]
	These maps are smooth and inverse to one another.
\end{proof}

The group $G$ acts on $\cT$ by conjugation, $g\cdot A=gAg^{-1}$. Its orbits correspond exactly to the isometry classes of the induced classical forms \(b(A\cdot,\cdot)\).
By the standard classification of classical forms \cite{Har,K}, this action is transitive except in the following four signature-dependent cases:
\[
\boxed{
	\begin{array}{c|c}
		\text{family} & G \\
		\hline
		\text{complex-structure, }\eps=1 & \SpG(2n,\R)
		\\ [1 mm] \hline
		\text{quaternionic-structure} & \SpG(2n,\C)
		\\ [1 mm] \hline
		\text{real-form} & \OO(n,\C)
		\\ [1 mm] \hline
		\text{complex-form, }\eps=1 & \Ostar(2n)
	\end{array}
}
\]
In each exceptional case, the signature strata $\cT_{p,q}$ defined in the corresponding subsection are precisely the $G$-orbits.

For the reference structure \(\sigma \in \cT \), its stabilizer is
\[
G_\sigma=\{g\in G:g\sigma=\sigma g\}=Z_G(\sigma).
\]
Hence the $G$-orbit of $\sigma$ (i.e. $\cT$ or $\cT_{p,q}$) is diffeomorphic to $G/G_\sigma$.

\begin{prop}\label{ss}
	$G/G_\sigma$ is a pseudo-Riemannian symmetric space.
\end{prop}
\begin{proof}
	On the Lie group $G$, we consider 
	\[
	\rho_\sigma: G \to G, \qquad g \mapsto \sigma g\sigma^{-1},
	\]
	Since $\sigma^2=\pm I$, $\rho_\sigma$ is an involutive automorphism. And
	\[
	G^{\rho_\sigma}=\{g\in G:g\sigma=\sigma g\}=G_\sigma,
	\]
	$(G,G_\sigma)$ is a symmetric pair,  $G/G_\sigma$ is a pseudo-Riemannian symmetric space.
\end{proof}

\subsection{Type III.1: Complex-structure Grassmannians}\label{sec:complex-structures}
Let $(V \cong \R^{2n}, b)$ be a classical space as before, i.e.
\[
b=0, ~  ~ b~  \text{ symmetric}, ~ ~  \text{or} ~ ~   b~   \text{skew-symmetric}.
\]
Suppose that \(V\) is equipped with a complex structure, that is, a
$\R$-linear map\footnote{A standard explicit realization of such $\sigma$ is given by $\sigma(u_1, u_2)=(-u_2, u_1)$ for any $(u_1, u_2) \in V$.}
\[
\sigma:V\longrightarrow V,\qquad
\sigma(\lambda v)=\lambda \sigma (v),\qquad
\sigma^2=-I.
\]
When \(b\ne0\), suppose moreover that
\[
b(\sigma u,\sigma v)=\eps b(u,v) \qquad \text{for}  \qquad \eps=\pm 1.
\]
Explicitly, the possibilities for $G$ are as follows\footnote{The explicit matrix realizations for these groups are listed as follows:
	(1) $b=0$; (2) $b(u,v)=u^t \operatorname{diag}(I_{p,q},I_{p,q})v$; (3) $b(u,v)=u^t J_{n,n} v$; (4) $b(u,v)=u^t \operatorname{anti-diag}(I,I) v$; (5) $b(u,v)=u^t \operatorname{anti-diag}(J_{m,m},J_{m,m})v$.}:
\begin{enumerate}[label=\textup{(\arabic*)}]
	\item $G=\GL(2n,\R)$ for $b=0$;
	\item $G=\OO(2p,2q)$ for $b$ symmetric and $\eps=1$, where $p+q=n$;
	\item $G=\SpG(2n,\R)$ for $b$ skew-symmetric and $\eps=1$;
	\item $G=\OO(n,n)$ for $b$ symmetric and $\eps=-1$;
	\item $G=\SpG(4m,\R)$ for $b$ skew-symmetric and $\eps=-1$, where $2m=n$.
\end{enumerate}

\bigskip

By Lemma \ref{R}, we obtain an involution $R_\sigma$ on $\olG$. We now compute the full fixed-point set of $R_\sigma$ on $\olG$, i.e. $\olG^{R_\sigma}$. On $E$, the linear operator $J\triangleq R_\sigma \colon E\to E$, equips $E$ with a complex vector-space structure. Write
\[
\C=\R \oplus \R \mathbf{i}, ~~ \text{with} ~  \mathbf{i}^2=-1.
\] 
For any $z \in E$, we set
\[
\mathbf{i} \cdot z := Jz.
\]
A  subspace $W\subset E$ is a complex subspace if and only if $W$ is $J$-invariant.

When $b\neq 0$, we define the form $\beta\colon E\times E\to\mathbb C$ by
\[
\beta(z,w):=B(z,w)-\mathbf{i}\eps B(Jz,w), ~ \forall z,w \in E.
\]

\begin{lem}
	The map $\beta$ is complex-bilinear when $\eps=1$  and sesquilinear when $\eps=-1$, i.e., for all $\lambda \in\mathbb C$, we have
	\[
	\beta(\lambda \cdot z,\,w)=\lambda \cdot \beta(z,w)  ~ when ~ \eps=1,\qquad 
	\beta(\lambda \cdot z,\,w)=\bar{\lambda} \cdot \beta(z,w)  ~ when ~ \eps=-1,
	\]
	\[
	\beta(z, \lambda \cdot w)=\lambda \cdot \beta(z,w).
	\]
	Moreover, for any $J$-invariant subspace $W\subset E$, one has the equivalence
	\[
	B|_W=0 \iff  \beta|_W=0.
	\]
\end{lem}
\begin{proof}
	For the first statement, it suffices to check for $\lambda=\mathbf{i}$. Note that 	$B(Jz,Jw)=- \eps B(z,w)$ and then $B(z,Jw)=\eps B(Jz,w)$, hence
	\[
	\beta(\mathbf{i} \cdot z,\,w)=B(Jz,w)-\mathbf{i}\eps B(J(Jz),w)=B(Jz,w)+\mathbf{i}\eps B(z,w)=\mathbf{i}\eps \beta(z,w),
	\]
	\[
	\beta(z,\,\mathbf{i} \cdot w)=B(z,Jw)-\mathbf{i}\eps B(Jz, Jw)=\eps B(Jz,w)+\mathbf{i} B(z,w)=\mathbf{i} \beta(z,w).
	\]
	
	The equivalence follows directly from the definition and the $J$-invariance of $W$.
\end{proof}

In fact, direct computation yields the following observations:
\begin{enumerate}[label=\textup{(\arabic*)}]
	\item   $b$  symmetric, $\eps=1$ $\implies$ $\beta$   complex symmetric;
	\item  $b$ skew-symmetric, $\eps=1$ $\implies$  $\beta$   complex skew-symmetric;
	\item  $b$ symmetric, $\eps=-1$ $\implies$ $\beta$   complex Hermitian with signature $(n,n)$;
	\item $b$ skew-symmetric, $\eps=-1$ $\implies$ $\beta$  complex skew-Hermitian with signature $(n,n)$.
\end{enumerate}
Note that all the above $\beta$ are split forms. Accordingly, for each case we obtain the associated Grassmannian
$\IG_{\C}(E,\beta)$, which consists of all maximal  $\beta$-isotropic subspaces of $E$. From the above lemma and discussion, we have

\begin{prop}
	The fixed-point set of $R_\sigma$ on $\olG$ is given by
	\[
	\Gr_{\R}(2n,E)^{R_\sigma}\cong  \Gr_{\C}(n,E), ~~~~ \IG_{\R}(E,B)^{R_\sigma}\cong \IG_{\C}(E,\beta).
	\]
\end{prop}

\bigskip

Now we consider the   graph-like locus of $\olG^{R_\sigma}$. By Proposition \ref{X-T}, we have
\[
\cX= \{ W \in \olGo: R_{\sigma}(W)=W\} \cong \left\{
A:V\to V \left|
\begin{array}{l}
	A\text{ is $\R$-linear},\\
	A^2=-I,\\
	b(Au,Av)=\eps b(u,v)
\end{array} \right.
\right\}=\cT.
\]
By the preceding $G$-action discussion, the action is transitive except for $G=\SpG(2n,\mathbb R)$, $\eps=1$. Here, $h_A(u,v):=b(Au,v)$ is symmetric, with signature $(2p,2q)$ labels the orbits: $\cT=\bigsqcup_{p+q=n}\cT_{2p,2q}$. By Proposition \ref{ss}, these orbits are pseudo-Riemannian symmetric spaces, and
\begin{thm}[Complex-structure Grassmannians]\label{thm:Complex-structure Grassmannian}
	We call these orbits \textbf{complex-structure Grassmannians}, as they parameterize  complex structures on \(V\) that are compatible with  \(b\). Explicitly, this construction produces the following five series:
	\[
	\frac{\GL(2n,\mathbb{R})}{\GL(n,\mathbb{C})},\;
	\frac{\OO(2p,2q)}{\UU(p,q)},\;
	\frac{\SpG(2n,\mathbb{R})}{\UU(p,q)},\;
	\frac{\OO(n,n)}{\OO(n,\mathbb{C})},\;
	\frac{\SpG(4m,\mathbb{R})}{\SpG(2m,\mathbb{C})}.
	\]
\end{thm}

\subsection{Type III.2: Quaternionic-structure Grassmannians}\label{sec:quaternion-structures}
Let \((V\cong \mathbb C^{2n},b)\) be a classical space as before, i.e. 
\[
b=0,~ 
b\text{ symmetric}, ~ 
b\text{  skew-symmetric}, ~ 
b\text{ Hermitian}, ~ \text{or} ~ 
b\text{ skew-Hermitian}.
\]
Suppose that \(V\) is equipped with a quaternionic structure, that is, a
conjugate-linear map\footnote{A standard explicit realization of such $\sigma$ is given by $\sigma(u_1, u_2)=(-\overline{u_2},\overline{u_1})$ for any $(u_1, u_2) \in V$.}
\[
\sigma:V\longrightarrow V,\qquad
\sigma(\lambda v)=\overline\lambda\,\sigma (v),\qquad
\sigma^2=-I.
\]
When \(b\ne0\), suppose moreover that
\[
b(\sigma u, \sigma v)=\overline{b(u,v)}.
\]
Explicitly, the possibilities for $G$ are as follows\footnote{The explicit matrix realizations for these groups are listed as follows:
	(1) $b=0$; (2) $b(u,v)=u^tv$; (3) $b(u,v)=u^t J_{n,n} v $; (4) $b(u,v)=\bar{u}^t \operatorname{diag}(I_{p,q}, I_{p,q})v$; (5) $b(u,v)=\bar{u}^t J_{n,n} v$.}:
\begin{enumerate}[label=\textup{(\arabic*)}]
	\item $G=\GL(2n,\C)$ for $b=0$;
	\item $G=\OO(2n,\C)$ for $b$ symmetric;
	\item $G=\SpG(2n,\C)$ for $b$ skew-symmetric;
	\item $G=\UU(2p,2q)$ for $b$ Hermitian, where $p+q=n$;
	\item $G=\UU(n,n)$ for $b$ skew-Hermitian.
\end{enumerate}

\bigskip

By Lemma \ref{R}, we obtain an involution $R_\sigma$ on $\olG$. We now compute the full fixed-point set of $R_\sigma$ on $\olG$, i.e. $\olG^{R_\sigma}$. The original complex scalar multiplication on $E$, together with the conjugate-linear operator
$J\triangleq R_\sigma \colon E\to E$, equips $E$ with a right quaternionic vector-space structure. Write
\[
\HH=\C \oplus \C \mathbf{j}, ~~ \text{with} ~  \mathbf{j}^2=-1, ~
\lambda  \mathbf{j}= \mathbf{j} \overline\lambda,  ~ \forall  \lambda \in \C.
\] 
For any $z \in E$, we set
\[
z\cdot \mathbf{j} := Jz.
\]
A complex subspace $W\subset E$ is a right quaternionic subspace if and only if $W$ is $J$-invariant.

When $b\neq 0$, we construct a complex sesquilinear form $C\colon E\times E\to\mathbb C$ satisfying
the key compatibility relation
\[
C(Jz,Jw)=\overline{C(z,w)},\quad \forall z,w\in E.
\]
The construction of $C$ depends on the type of $B$:
\begin{itemize}
	\item If $B$ is complex-bilinear (i.e., $b$ is symmetric or skew-symmetric), define
	\[
	C(z,w):=\mathbf{i}B(Jz,w).
	\]
	\item If $B$ is sesquilinear (i.e., $b$ is Hermitian or skew-Hermitian), define
	\[
	C(z,w):=\mathbf{i}B(z,w).
	\]
\end{itemize}

Now define the form $\beta\colon E\times E\to\mathbb H$ by
\[
\beta(z,w):=C(z,w)-C(z,Jw)\,\mathbf j,\quad \forall z,w\in E.
\]

\begin{lem}
	The map $\beta$ is a right quaternionic sesquilinear form, i.e., for all $q\in\mathbb H$,
	\[
	\beta(z\cdot q,\,w)= \overline{q} \cdot \beta(z,w),\qquad
	\beta(z,\,w\cdot q)=\beta(z,w)\cdot q.
	\]
	Moreover, for any $J$-invariant complex subspace $W\subset E$, one has the equivalence
	\[
	B|_W=0 \iff C|_W=0 \iff \beta|_W=0.
	\]
\end{lem}

\begin{proof}
	Recall that $\mathbb H = \mathbb C \oplus \mathbb C\mathbf j$, so any quaternion can be written as
	$q=\lambda_1+\lambda_2\mathbf j$ with $\lambda_1,\lambda_2\in\mathbb C$.
	By real linearity, it suffices to verify the identities
	for $q\in\mathbb C$ and $q=\mathbf j$. We record only the case $q=\mathbf j$; the complex-scalar case is immediate.
	
	Since $C(Jz,Jw)=\overline{C(z,w)}$, $C(Jz,w)=-\overline{C(z,Jw)}$. 
	
	First variable:
	\[
	\beta(z\cdot\mathbf j,w)=\beta(Jz,w)=C(Jz,w)-C(Jz,Jw)\,\mathbf j=-\overline{C(z,Jw)}-\overline{C(z,w)}\,\mathbf j
	\] 
	\[
	\overline{\mathbf j} \cdot \beta(z,w)=\overline{\mathbf j} \cdot(C(z,w)-C(z,Jw)\,\mathbf j)=-\overline{C(z,w)}\,\mathbf j-\overline{C(z,Jw)}=\beta(z\cdot\mathbf j,w)
	\]
	
	Second variable:
	
	\[
	\beta(z,w\cdot\mathbf j)=\beta(z,Jw)=C(z,Jw)+C(z,w)\,\mathbf j
	\]
	\[
	\beta(z,w)\cdot\mathbf j=(C(z,w)-C(z,Jw)\,\mathbf j) \cdot\mathbf j=C(z,w)\,\mathbf j+C(z,Jw)=\beta(z,w\cdot\mathbf j)
	\]
	
	Since $W$ is $J$-invariant, the definitions of $C$ and $\beta$ give the equivalence.
\end{proof}

In fact, direct computation yields the following observations:
\begin{enumerate}[label=\textup{(\arabic*)}]
	\item   $b$  symmetric $\implies$ $\beta$ quaternionic skew-Hermitian;
	\item   $b$  skew-symmetric $\implies$ $\beta$ quaternionic Hermitian with signature $(n,n)$;
	\item  $b$  Hermitian $\implies$ $\beta$ quaternionic skew-Hermitian;
	\item  $b$  skew-Hermitian $\implies$ $\beta$ quaternionic Hermitian with signature $(n,n)$.
\end{enumerate}
Note that all the above $\beta$ are split forms.  From the above lemma and discussion, we have

\begin{prop}
	The fixed-point set of $R_\sigma$ on $\olG$ is given by
	\[
	\Gr_{\C}(2n,E)^{R_\sigma}\cong  \Gr_{\HH}(n,E), ~~~~    \IG_{\mathbb C}(E,B)^{R_\sigma}\cong \IG_{\HH}(E,\beta).
	\]
\end{prop}

\bigskip

Now we consider the  graph-like locus of $\olG^{R_\sigma}$. By Proposition \ref{X-T}, we have
\[
\cX= \{ W \in \olGo: R_{\sigma}(W)=W\} \cong \left\{
A:V\to V \left|
\begin{array}{l}
	A\text{ is conjugate-linear},\\
	A^2=-I,\\
	b(Au,Av)=
	\overline{b(u,v)}
\end{array} \right.
\right\}=\cT.
\]
By the preceding $G$-action discussion, the action is transitive except for $G=\SpG(2n,\C)$. Here, $h_A(u,v):=b(Au,v)$ is Hermitian, with signature $(2p,2q)$ labels the orbits: $\cT=\bigsqcup_{p+q=n}\cT_{2p,2q}$. By Proposition \ref{ss}, these orbits are pseudo-Riemannian symmetric spaces, and
\begin{thm}[Quaternionic-structure Grassmannians]\label{thm:Quaternionic-structure Grassmannian}
	We call these orbits \textbf{quaternionic-structure Grassmannians}, as they parameterize  quaternionic structures on \(V\) that are compatible with  \(b\). Explicitly, this construction produces the following five series:
	\[
	\frac{\GL(2n,\mathbb{C})}{\GL(n,\mathbb{H})},\;
	\frac{\OO(2n,\mathbb{C})}{\Ostar(2n)},\;
	\frac{\SpG(2n,\mathbb{C})}{\SpG(p,q)},\;
	\frac{\UU(2p,2q)}{\SpG(p,q)},\;
	\frac{\UU(n,n)}{\Ostar(2n)}.
	\]
\end{thm}

\subsection{Type III.3: Real-form Grassmannians}\label{sec:real-forms}
Let \((V\cong \mathbb C^{n},b)\) be a classical space as before, i.e. 
\[
b=0,~ 
b\text{ symmetric}, ~ 
b\text{  skew-symmetric}, ~ 
b\text{ Hermitian}, ~ \text{or} ~ 
b\text{ skew-Hermitian}.
\]
Suppose that \(V\) is equipped with a real-form structure, that is, a
conjugate-linear map\footnote{A standard explicit realization of such $\sigma$ is given by $\sigma(u)=\overline{u}$ for any $u \in V$.}
\[
\sigma:V\longrightarrow V,\qquad
\sigma(\lambda v)=\overline\lambda\,\sigma(v),\qquad
\sigma^2=I.
\]
When \(b\ne0\), suppose moreover that
\[
b(\sigma u,\sigma v)=\overline{b(u,v)}.
\]
Explicitly, the possibilities for $G$ are as follows\footnote{The explicit matrix realizations for these groups are listed as follows:
	(1) $b=0$; (2) $b(u,v)=u^tv$; (3) $b(u,v)=u^t J_{m,m} v $; (4) $b(u,v)=\bar{u}^t I_{p,q}v$; (5) $b(u,v)=\bar{u}^t J_{m,m} v$.}:
\begin{enumerate}[label=\textup{(\arabic*)}]
	\item $G=\GL(n,\C)$ for $b=0$;
	\item $G=\OO(n,\C)$ for $b$ symmetric;
	\item $G=\SpG(2m,\C)$ for $b$ skew-symmetric, here $2m=n$;
	\item $G=\UU(p, q)$ for $b$ Hermitian, where $p+q=n$;
	\item $G=\UU(m,m)$ for $b$ skew-Hermitian, where $2m=n$.
\end{enumerate}

\bigskip

By Lemma \ref{R}, we obtain an involution $R_\sigma$ on $\olG$. We now compute the full fixed-point set of $R_\sigma$ on $\olG$, i.e. $\olG^{R_\sigma}$.
Set $L:=E^{R_{\sigma}}=\{z \in E: R_{\sigma}(z)=z\}$, i.e.
\[
L=\Graph(\sigma)=\{(v,\sigma(v)):v\in V\}\subset E,
\]
then its opposite graph is
\[
\mathbf i L =\Graph(-\sigma)=\{(v,-\sigma(v)):v\in V\} \subset E,
\]
and  $E=L\oplus \mathbf i L$.  Hence $R_{\sigma}$ is the identity on $L$ and minus the identity on $\mathbf i L$.

\begin{prop}
	The fixed-point set of $R_{\sigma}$ on $\olG$ is given by
	\[
	\Gr_{\C}(n,E)^{R_\sigma}\cong  \Gr_{\R}(n,L), ~~~~     \IG_{\mathbb C}(E,B)^{R_\sigma}\cong \IG_{\R}(L,\beta),
	\]
	where $\beta=-\mathbf{i}B|_L$ is a real-valued form on $L$.
\end{prop}
\begin{proof}
	Let $W \subset E$ be a complex $n$-plane with $R_{\sigma}(W)=W$. Then $W=(W\cap L) \oplus (W\cap \mathbf i L)$.
	Set $W_{\R}:=W\cap L$. Then $W\cap \mathbf i L= \mathbf i W_{\R}$, $W=W_{\R} \oplus \mathbf i W_{\R}$.
	Thus $W\mapsto W_{\R}$ gives a bijection from $\Gr_{\C}(n,E)^{R_\sigma}$ to $\Gr_{\R}(n,L)$. Hence
	\[
	\Gr_{\C}(n,E)^{R_\sigma}\cong  \Gr_{\R}(n,L).
	\]

	When $b\neq 0$, for $z,w\in L$, one has $B(z,w)=-\overline{B(z,w)}$. Explicitly,  for any $z=(x, \sigma(x)), w=(y, \sigma(y)) \in L$, we have
	\[
	B((x, \sigma(x)), (y, \sigma(y)))=b(x,y)-b(\sigma(x),\sigma(y))=b(x,y)-\overline{b(x,y)}=2\mathbf i \operatorname{Im}b(x,y) 
	\]
	Thus $\beta:=-\mathbf{i}B|_L$ is a real-valued form on $L$. And we have 
	\[
	B|_W=0 \iff B|_{W_{\R}}=0 \iff \beta|_{W_{\R}}=0.
	\]
	Consequently,
	\[
	\IG_{\mathbb C}(E,B)^{R_\sigma}\cong \IG_{\R}(L,\beta).
	\]
\end{proof}

Note that  direct computation yields the following observations:
\begin{enumerate}[label=\textup{(\arabic*)}]
	\item $b$   symmetric  $\implies$ $\beta$ real  symmetric of signature $(n,n)$;
	\item $b$   skew-symmetric  $\implies$  $\beta$ real skew-symmetric;
	\item $b$   Hermitian  $\implies$  $\beta$ real skew-symmetric;
	\item $b$   skew-Hermitian  $\implies$ $\beta$ real symmetric of signature $(n,n)$.
\end{enumerate}
All these $\beta$ are split. Hence, $\IG_{\R}(L,\beta)$ within the above proposition is significant.

\bigskip

Now we consider the  graph-like locus of $\olG^{R_\sigma}$. By Proposition \ref{X-T}, we have
\[
\cX= \{ W \in \olGo: R_{\sigma}(W)=W\} \cong \left\{
A:V\to V \left|
\begin{array}{l}
	A\text{ is conjugate-linear},\\
	A^2=I,\\
	b(Au,Av)=
	\overline{b(u,v)}
\end{array} \right.
\right\}=\cT.
\]
By the preceding $G$-action discussion, the action is transitive except for $G=\OO(n,\C)$. Here, $h_A(u,v):=b(Au,v)$ is Hermitian, with signature $(p,q)$ labels the orbits: $\cT=\bigsqcup_{p+q=n}\cT_{p,q}$. By Proposition \ref{ss}, these orbits are pseudo-Riemannian symmetric spaces, and
\begin{thm}[Real-form Grassmannians]\label{thm:Real-form Grassmannian}
	We call these orbits \textbf{real-form Grassmannians}, as they parameterize  real-form structures on \(V\)
	that are compatible with \(b\). Explicitly, this construction produces the following five series:
	\[
	\frac{\GL(n,\mathbb{C})}{\GL(n,\mathbb{R})},\;
	\frac{\OO(n,\mathbb{C})}{\OO(p,q)},\;
	\frac{\SpG(2m,\mathbb{C})}{\SpG(2m,\mathbb{R})},\;
	\frac{\UU(p,q)}{\OO(p,q)},\;
	\frac{\UU(m,m)}{\SpG(2m,\mathbb{R})}.
	\]
\end{thm}

\subsection{Type III.4: Complex-form Grassmannians}\label{sec:complex-forms}
Let $(V \cong \HH^{n}, b)$ be a classical space as before, i.e.
\[
b=0, ~  ~ b~  \text{Hermitian}, ~ ~  \text{or} ~ ~   b~   \text{skew-Hermitian}.
\]
Suppose that \(V\) is equipped with a complex-form structure, that is, a
$\HH$-linear map\footnote{A standard explicit realization of such $\sigma$ is given by left multiplication  $\sigma(u)=\mathbf i \cdot u$ for any $u \in V$.}
\[
\sigma:V\longrightarrow V,\qquad
\sigma(vq)=\sigma(v)q,\qquad
\sigma^2=-I.
\]
When \(b\ne0\), suppose moreover that
\[
b(\sigma u ,\sigma v )=\eps b(u,v) \qquad \text{for}  \qquad \eps=\pm 1.
\]
Explicitly, the possibilities for $G$ are as follows\footnote{The explicit matrix realizations for these groups are listed as follows:
	(1) $b=0$; (2) $b(u,v)=\bar{u}^t I_{p,q}v$; (3) $b(u,v)=\bar{u}^t \mathbf i v$; (4) $b(u,v)=\bar{u}^t \mathbf j J_{m,m}v$; (5) $b(u,v)=\bar{u}^t \mathbf j v$.}:
\begin{enumerate}[label=\textup{(\arabic*)}]
	\item $G=\GL(n,\HH)$ for $b=0$;
	\item $G=\SpG(p,q)$ for  $b$ Hermitian and $\eps=1$, here $p+q=n$;
	\item $G=\Ostar(2n)$ for $b$ skew-Hermitian and $\eps=1$;
	\item $G=\SpG(m,m)$ for $b$ Hermitian and $\eps=-1$, here $2m=n$;
	\item $G=\Ostar(2n)$ for $b$ skew-Hermitian and $\eps=-1$.
\end{enumerate}

\bigskip

By Lemma \ref{R}, we obtain an involution $R_\sigma$ on $\olG$. We now compute the full fixed-point set of $R_\sigma$ on $\olG$, i.e. $\olG^{R_\sigma}$. Set
\[
E^+:=\{z \in E: R_{\sigma}(z)=z \cdot \mathbf{i}\} \subset E,
\]
\[
E^-:=\{z \in E: R_{\sigma}(z)=- z \cdot \mathbf{i}\} \subset E,
\]
then $E^+$ and $E^-$ are complex vector spaces, $E^-=E^+ \cdot \mathbf{j}$ and $E=E^+ \oplus E^-=E^+ \oplus E^+ \cdot \mathbf{j}$.

\begin{prop}
	The fixed-point set of $R_{\sigma}$ on $\olG$ is given by
	\[
	\Gr_{\HH}(n,E)^{R_\sigma}\cong  \Gr_{\C}(n,E^+), ~~~~    \IG_{\mathbb H}(E,B)^{R_\sigma}\cong \IG_{\C}(E^+,\beta),
	\]
	where $\beta$ denotes the complex-valued form on $E^+$ defined by $\beta=\mathbf{j}B|_{E^+}$ for $\eps=1$ and $\beta=B|_{E^+}$ for $\eps=-1$.
\end{prop}
\begin{proof}
	Let $W \subset E$ be a quaternionic $n$-plane with $R_{\sigma}(W)=W$. Then $W=(W\cap E^+) \oplus (W\cap E^-)$.
	Set $W^+:=W\cap E^+$. Then $W\cap E^-=(W\cap E^+) \mathbf{j}$, $W=W^+\oplus W^+ \mathbf{j}$.
	Thus $W\mapsto W^+$ gives a bijection from $\Gr_{\HH}(n,E)^{R_\sigma}$ to $\Gr_{\C}(n,E^+)$. Hence
	\[
	\Gr_{\HH}(n,E)^{R_\sigma}\cong  \Gr_{\C}(n,E^+).
	\]

	When $b\neq 0$, for any $z,w \in E^+$, we have 
	\[
	B(R_{\sigma}(z), R_{\sigma}(w))=B(z \mathbf i, w \mathbf i)=- \mathbf{i} B(z,w) \mathbf{i}.
	\]
	
	If  $\eps=1$, we have $	B(R_{\sigma}(z), R_{\sigma}(w))=-B(z,w)$, hence $\beta:=\mathbf{j}B|_{E^+}$ is a complex form on $E^+$.
	
	If  $\eps=-1$, we have $B(R_{\sigma}(z), R_{\sigma}(w))=B(z,w)$, hence $\beta:=B|_{E^+}$ is a complex form on $E^+$.
	
	In either case, we have 
	\[
	B|_W=0 \iff \beta|_{W^+}=0.
	\]
	Consequently,
	\[
	\IG_{\mathbb H}(E,B)^{R_\sigma}\cong \IG_{\C}(E^+,\beta).
	\]
\end{proof}

Note that  direct computation yields the following observations:
\begin{enumerate}[label=\textup{(\arabic*)}]
	\item    $b$   Hermitian,  $\eps=1$  $\implies$ $\beta$ complex  skew-symmetric;
	\item    $b$   skew-Hermitian,  $\eps=1$  $\implies$  $\beta$ complex symmetric;
	\item   $b$   Hermitian,  $\eps=-1$  $\implies$  $\beta$ complex Hermitian of signature $(n,n)$;
	\item   $b$   skew-Hermitian,  $\eps=-1$  $\implies$ $\beta$ complex skew-Hermitian of signature $(n,n)$.
\end{enumerate}
All these $\beta$ are split. Hence, $ \IG_{\C}(E^+,\beta)$ within the above proposition is significant.

\bigskip

Now we consider the  graph-like locus of $\olG^{R_\sigma}$. By Proposition \ref{X-T}, we have
\[
\cX= \{ W \in \olGo: R_{\sigma}(W)=W\} \cong \left\{
A:V\to V \left|
\begin{array}{l}
	A\text{ is $\HH$-linear},\\
	A^2=-I,\\
	b(Au,Av)=\eps b(u,v)
\end{array} \right.
\right\}=\cT.
\]
By the preceding $G$-action discussion, the action is transitive except for $G=\Ostar(2n)$, $\eps=1$. Here, $h_A(u,v):=b(Au,v)$ is Hermitian, with signature $(p,q)$ labels the orbits: $\cT=\bigsqcup_{p+q=n}\cT_{p,q}$. By Proposition \ref{ss}, these orbits are pseudo-Riemannian symmetric spaces, and
\begin{thm}[Complex-form Grassmannians]\label{thm:Complex-form Grassmannian}
	We call these orbits \textbf{complex-form Grassmannians}, as they parameterize  complex-form structures on \(V\)
	that are compatible with \(b\). Explicitly, this construction produces the following five series:
	\[
	\frac{\GL(n,\mathbb{H})}{\GL(n,\mathbb{C})},\;
	\frac{\SpG(p,q)}{\UU(p,q)},\;
	\frac{\Ostar(2n)}{\UU(p,q)},\;
	\frac{\SpG(m,m)}{\SpG(2m,\mathbb{C})},\;
	\frac{\Ostar(2n)}{\OO(n,\mathbb{C})}.
	\]
\end{thm}

\bigskip
\section{Type IV: Subspace-configuration Grassmannians}\label{sec:subspace}
We treat three families of Grassmannians parametrizing subspace configurations in this section. 

\subsection{Type IV.1: Isotropic-pair Grassmannians}\label{sec:b-isotropic}
Let $(V \cong \F^{2n},b)$ be a classical space with $b$ a split form.  Thus $V$ admits a decomposition
\[
V=V_+\oplus V_-,
\]
where $V_+$ and $V_-$ are maximal $b$-isotropic subspaces of $(V,b)$.  Define
\[
s:V\to V,
\qquad s(v_++v_-)=v_+-v_-,\qquad v_{\pm}\in V_{\pm}.
\]
Since $b|_{V_+}=b|_{V_-}=0$, the map $s$ is a $b$-anti-isometry:
\begin{equation}\label{eq:s-anti-isometry}
	b(su,sv)=-b(u,v),\qquad u,v\in V.
\end{equation}
Put $E=V\oplus V$ and $B=b\oplus(-b)$.  The graph
\[
L=\Graph(s)=\{(v,sv):v\in V\}\subset E
\]
is a nondegenerate subspace of $(E,B)$.  Its natural opposite graph is
\[
L^- =\Graph(-s)=\{(v,-sv):v\in V\} \subset E.
\]
Then $E=L\oplus L^-$.  We define the reflection along $L$ to be the involution
\[
R_L:E\to E,
\qquad R_L(x,y)=(sy,sx).
\]
Equivalently, $R_L$ is the identity on $L$ and minus the identity on $L^-$. 

Denote $G=Aut(V,b)$, $\olG=\IG(E,B)$ as before, then
\begin{lem}\label{lem:reflection-orthogonal}
	The map $R_L$ induces an involution on $\olG$.
\end{lem}
\begin{proof}
	Using \eqref{eq:s-anti-isometry}, for $(x,y),(x',y')\in E$ we have
	\[
	\begin{aligned}
		B(R_L(x,y),R_L(x',y'))
		&=B((sy,sx),(sy',sx')) \\
		&=b(sy,sy')-b(sx,sx') \\
		&=-b(y,y')+b(x,x') \\
		&=B((x,y),(x',y')).
	\end{aligned}
	\]
	It follows that $R_L$ sends maximal $B$-isotropic subspaces to maximal $B$-isotropic subspaces. 
\end{proof}

\begin{prop}\label{prop:full-fixed-locus}
	The fixed-point set of $R_L$ on $\olG=\IG(E,B)$ is
	\[
	\IG(E,B)^{R_L}
	\cong
	\IG(L,B|_L)\times \IG(L^-,B|_{L^-}).
	\]
	Under the identifications $L\cong V$ and $L^-\cong V$, this becomes
	\[
	\IG(E,B)^{R_L}\cong \IG(V,b)\times \IG(V,b).
	\]
\end{prop}

\begin{proof}
	Let $W\in \IG(E,B)$ be fixed by $R_L$.  Since $R_L$ is diagonalizable with $+1$-eigenspace $L$ and $-1$-eigenspace $L^-$, the $R_L$-invariant subspace $W$ decomposes as
	\[
	W=(W\cap L)\oplus(W\cap L^-).
	\]
	Write $A=W\cap L$ and $C=W\cap L^-$.  The summands $L$ and $L^-$ are $B$-orthogonal, and $B|_L$ and $B|_{L^-}$ are both isometric to a non-zero scalar multiple of $b$.  Therefore $W$ is $B$-isotropic if and only if $A$ is isotropic in $L$ and $C$ is isotropic in $L^-$.  Since $W$ is maximal $B$-isotropic of dimension $2n$ and the maximal isotropic dimension in each of $L$ and $L^-$ is $n$, we must have
	\[
	\dim_{\F} A=\dim_{\F} C=n.
	\]
	Thus $A$ and $C$ are maximal isotropic subspaces of $L$ and $L^-$, respectively.
	
	Conversely, if $A\subset L$ and $C\subset L^-$ are maximal isotropic, then $A\oplus C$ is a maximal $B$-isotropic subspace of $E$ and is fixed by $R_L$.  This proves the asserted product description.

	Finally, the maps $v\mapsto(v,sv)$ and $v\mapsto(v,-sv)$ identify $(V,b)$ with $(L,\frac12 B|_L)$ and $(L^-,\frac12 B|_{L^-})$, respectively.
\end{proof}

Now we consider the   graph-like locus of $\olG^{R_L}$:
\begin{equation}
	\cX:=\olG^{R_L} \cap \olGo.
\end{equation}
Also define
\begin{equation}\label{eq:Tb-def}
	\cT=\bigl\{(A,C)\in \IG(V,b)\times\IG(V,b):A\oplus C=V\bigr\},
\end{equation}
i.e. the space of ordered transverse maximal $b$-isotropic pairs in $V$.

\begin{prop}\label{prop:Xb-Tb}
	There is a natural bijective smooth identification
	\[
	\cX\cong \cT.
	\]
\end{prop}

\begin{proof}
	By Proposition \ref{prop:full-fixed-locus}, any $W \in \cX$ has the form
	\[
	W=A\oplus C,
	\]
	where $A\subset L$ and $C\subset L^-$ are maximal isotropic.  Now identify $(L,\frac12 B|_L)$ and $(L^-,\frac12 B|_{L^-})$ with $(V,b)$ as before, then $A$ and $C$ correspond to maximal $b$-isotropic subspaces $A_0, C_0 \subset V$. Thus
	\[
	W=\{(a+c,sa-sc):a\in A_0,\ c\in C_0\}.
	\]
	
	By direct computation, we have
	\[
	W \cap(V\oplus \{0\})=\{0\}
	\Longleftrightarrow W \cap(\{0\}\oplus V)=\{0\}\Longleftrightarrow
	A_0\cap C_0=\{0\} \Longleftrightarrow A_0\oplus C_0=V.
	\]
	This proves the claim.
\end{proof}

\begin{lem}\label{lem:transverse-pairs}
	The group $G = \Aut(V,b)$  acts transitively on the set $\cT$. Moreover, the stabilizer subgroup of the reference pair $(V_+,V_-)$ admits a natural isomorphism to $\GL(V_+)$, which is isomorphic to $\GL(n,\mathbb{F})$.
\end{lem}

\begin{proof}
	Take an arbitrary pair $(A,C) \in \cT$. Select a basis $\{a_1,\dots,a_n\}$ for the subspace $A$. Since $A$ and $C$ are transverse maximal isotropic subspaces, the pairing
	\[
	\begin{aligned}
		A \times C &\to \mathbb{F}, \\
		(a,c) &\mapsto b(a,c)
	\end{aligned}
	\]
	is nondegenerate. 
	Choose a dual basis $\{c_1,\dots,c_n\}$ of $C$ satisfying the canonical orthogonality relations between the two bases. As both $A$ and $C$ are isotropic subspaces, the ordered basis
	\[
	a_1,\dots,a_n,\,c_1,\dots,c_n
	\]
	has the same Gram matrix as the canonical split basis associated with the decomposition $V = V_+ \oplus V_-$.
	Define the linear transformation sending the canonical split basis to the constructed basis above; this map preserves $b$ and thus belongs to $G$. By construction, this element maps the reference pair $(V_+,V_-)$ onto $(A,C)$, so the $G$-action on $\cT$ is transitive.
	
	For the stabilizer, Any automorphism in $G$ that fixes the ordered pair $(V_+,V_-)$ must take the block-diagonal form
	\[
	\begin{pmatrix}
		A & 0 \\
		0 & D
	\end{pmatrix}.
	\]
	The invariance requirement for $b$ enforces that $D$ equals the inverse adjoint of $A$ with respect to the canonical pairing between $V_+$ and $V_-$. Hence the stabilizer is parametrized by $A\in\GL(V_+)$, yielding a group isomorphism $\operatorname{Stab}_G(V_+,V_-) \cong \GL(n,\mathbb{F})$.
\end{proof}

\begin{thm}[Isotropic-pair Grassmannians]\label{thm:Isotropic-pair Grassmannians}
	The manifolds $\cT$ are pseudo-Riemannian symmetric spaces.
	We call them \textbf{isotropic-pair Grassmannians}, as they parameterize  ordered pairs of transverse maximal $b$-isotropic subspaces in $V$. Explicitly, this construction produces the following seven series:
	\[
	\frac{\OO(n,n)}{\GL(n,\mathbb{R})},\;
	\frac{\UU(n,n)}{\GL(n,\mathbb{C})},\;
	\frac{\SpG(n,n)}{\GL(n,\mathbb{H})},\;
	\frac{\SpG(2n,\mathbb{R})}{\GL(n,\mathbb{R})},\;
	\frac{\SpG(2n,\mathbb{C})}{\GL(n,\mathbb{C})},\;
	\frac{\OO(2n,\mathbb{C})}{\GL(n,\mathbb{C})},\;
	\frac{\Ostar(4n)}{\GL(n,\mathbb{H})}.
	\]
\end{thm}
\begin{proof}
	By Lemma \ref{lem:transverse-pairs}, $G$ acts transitively on $\cT$ with stabilizer $\GL(n,\F)$.  Hence
	\[
	\cT\cong G/\GL(n,\F).
	\]
	The seven displayed spaces are exactly the possibilities for the classical split forms in Table \ref{tab:IG}.  They are pseudo-Riemannian symmetric spaces because the stabilizer $\GL(n,\F)$ is the fixed-point subgroup, up to connected components, of an involution of $G$. In a split basis, this involution is
	\[
	g\longmapsto SgS^{-1},
	\qquad S|_{V_+}=1,\quad S|_{V_-}=-1,
	\]
	with the evident interpretation in each classical group.
\end{proof}

\subsection{Type IV.2: Nondegenerate-subspace Grassmannians}\label{sec:b-nondegenerate}
Let $(V \cong \mathbb{F}^n, b)$ be a classical space with $b \neq 0$. For any subspace $V_1 \subset V$ on which the restriction $b|_{V_1}$ is nondegenerate, let $V_2 := V_1^{\perp_b}$ denote its $b$-orthogonal complement. We then have $\dim V_1 + \dim V_2 = n$ and $V_1 \cap V_2 = \{0\}$, which yields the direct sum decomposition $V = V_1 \oplus V_2$.
Define
\[
s:V\to V,
\qquad s(v_1+v_2)=v_1-v_2,\qquad v_{i}\in V_{i}.
\]
Since $V_1 \perp_b V_2$, the map $s$ is a $b$-isometry:
\begin{equation}\label{eq:s-isometry}
	b(su,sv)=b(u,v),\qquad u,v\in V.
\end{equation}
Put $E=V\oplus V$ and $B=b\oplus(-b)$.  The graph
\[
L=\Graph(s)=\{(v,sv):v\in V\}\subset E
\]
is a maximal isotropic  subspace of $(E,B)$.  Its natural opposite graph is
\[
L^- =\Graph(-s)=\{(v,-sv):v\in V\} \subset E.
\]
Then $E=L\oplus L^-$.  We define the reflection along $L$ to be the involution
\[
R_L:E\to E,
\qquad R_L(x,y)=(sy,sx).
\]
Equivalently, $R_L$ is the identity on $L$ and minus the identity on $L^-$.

Denote $G=Aut(V,b)$, $\olG=\IG(E,B)$ as before, then
\begin{lem}\label{lem:reflection-orthogonal 2}
	The map $R_L$ induces an involution on $\olG$.
\end{lem}
\begin{proof}
	Using \eqref{eq:s-isometry}, for $(x,y),(x',y')\in E$ we have
	\[
	\begin{aligned}
		B(R_L(x,y),R_L(x',y'))
		&=B((sy,sx),(sy',sx')) \\
		&=b(sy,sy')-b(sx,sx') \\
		&=b(y,y')-b(x,x') \\
		&=-B((x,y),(x',y')).
	\end{aligned}
	\]
	It follows that $R_L$ sends maximal $B$-isotropic subspaces to maximal $B$-isotropic subspaces. 
\end{proof}

\begin{prop}\label{prop:full-fixed-locus 2}
	The fixed-point set of $R_L$ on $\olG=\IG(E,B)$ is
	\[
	\IG(E,B)^{R_L}\cong  \bigsqcup_{k=0}^{n} \Gr_{\F}(k,V).
	\]
\end{prop}
\begin{proof}
	Let $W\in \IG(E,B)$ be fixed by $R_L$.  Since $R_L$ is diagonalizable with $+1$-eigenspace $L$ and $-1$-eigenspace $L^-$, the $R_L$-invariant subspace $W$ decomposes as
	\[
	W=(W\cap L)\oplus(W\cap L^-).
	\]
	Write $A=W\cap L$ and $C=W\cap L^-$.  
	Since $L$ and $L^-$ are themselves $B$-isotropic, and 
	\[
	B\big((a, sa), (c, -sc)\big)
	= b(a,c) - b(sa, -sc)
	= b(a,c) + b(a,c)
	= 2b(a,c),
	\]
	the condition that $W = A \oplus C$ be $B$-isotropic is exactly
	\[
	b(A_0, C_0) = 0, 
	\]
	where $A_0, C_0 \subset V$ are the corresponding subspaces of $A, C$ under the identification $L\cong V$, $L^-\cong V$.
	If $\dim_{\F} A_0 = k$, then $\dim_{\F} C_0 = n - k$. Since $b$ is nondegenerate,
	\[
	\dim_{\F} A_0^{\perp_b} = n - k.
	\]
	Thus $C_0 \subset A_0^{\perp_b}$
	and equal dimensions imply
	\[
	C_0 = A_0^{\perp_b}.
	\]
	
	Therefore the fixed locus inside $\IG(E,B)$ is parametrized by arbitrary subspaces $A \subset V$:
	\[
	W_A = \big\{(a + c,\, sa - sc) : a \in A,\; c \in  A^{\perp_b}\big\}.
	\]
	
	Hence
	\[
	\IG(E,B)^{R_L} \cong \bigsqcup_{k=0}^{n} \Gr_{\F}(k,V).
	\]
	
\end{proof}

From the above proposition, $\IG(E,B)^{R_L}$ has $n+1$ components, each of which is given by 
\[
\IG(E,B)^{R_L}_k=\{W \in \IG(E,B)^{R_L}: \dim_{\F}(W \cap L)=k\} \cong \Gr_{\F}(k,V).
\]

Define the $k$-th graph-like  locus
\begin{equation}\label{eq:Xk-def}
\cX_k:=\olG^{R_L}_k \cap \olGo
\end{equation}
Also define
\begin{equation}\label{eq:Tk-def}
	\cT_k=\bigl\{A \in \Gr_{\F}(k,V): b|_A \text{ is  nondegenerate}\bigr\},
\end{equation}
i.e.  the space of $k$-dimensional subspaces in $V$ on which  $b$ is nondegenerate.

\begin{prop}\label{prop:Xk-Tk}
	There is a natural bijective smooth identification
	\[
	\cX_k\cong \cT_k.
	\]
\end{prop}
\begin{proof}
	Using the parametrization in  \ref{prop:full-fixed-locus 2}, any $W \in  \cX_k$  has the form
	\[
	W=	W_A = \big\{(a + c,\, sa - sc) : a \in A,\; c \in  A^{\perp_b}\big\},
	\]
	where $A \subset V$ and $\dim_{\F} A = k$.	By direct computation, we have
	\[
	W \cap(V\oplus \{0\})=\{0\}
	\Longleftrightarrow W \cap(\{0\}\oplus V)=\{0\}\Longleftrightarrow
	A\cap A^{\perp_b}=\{0\} \Longleftrightarrow b|_A \text{ is  nondegenerate}.
	\]
	This proves the claim.
\end{proof}

For signature forms, we further have
\[
\cT_k=\bigsqcup_{p+q=k}\cT_{p,q}
\]
with
\[
\cT_{p,q} \cong \{A\in\Gr_{\F}(k,V):b|_A\text{ has signature }(p,q)\},
\]
where $p+q=k$ and the indices are compatible with the signature of $b$ on $V$.

\begin{lem}\label{lem:nondegenerate-normal-form}
	The group $G=\Aut(V,b)$ acts transitively on $\cT_k$ in the non-signature cases and on each $\cT_{p,q}$ in the signature cases.  The stabilizer of a reference subspace $A_0$ is the product of the automorphism groups of $(A_0,b|_{A_0})$ and $(A_0^{\perp_b},b|_{A_0^{\perp_b}})$, i.e. $\operatorname{Stab}_G(A_0)
	\cong
	\Aut(A_0,b|_{A_0})
	\times
	\Aut(A_0^{\perp_b},b|_{A_0^{\perp_b}})$.
\end{lem}

\begin{proof}
	For any $A \in 	\cT_k$, we have 
	\[
	V=A\oplus A^{\perp_b}.
	\]
	
	Now let $A,A'\in \cT_k$ and suppose that $(A,b|_A)$ and
	$(A',b|_{A'})$ are isometric. By Witt's extension theorem, every
	isometry
	\[
	\phi:(A,b|_A)\longrightarrow (A',b|_{A'})
	\]
	extends to an element $g\in \Aut(V,b)$. Therefore
	$gA=A'$. Hence $G$-orbits in $\cT_k$ are exactly the isometry classes
	of the restricted forms $b|_A$. Thus $G$ acts transitively on $\cT_k$ when $b$ is not a signature form.
	In the signature form cases, the isometry class of $b|_A$ is determined
	by its signature $(p,q)$. Therefore $G$ acts transitively on each
	$\cT_{p,q}$.

	Let $g\in G$ stabilize $A_0$. Then for $a\in A_0$ and $c\in A_0^{\perp_b}$,
	\[
	b(ga,gc)=b(a,c)=0.
	\]
	Since \(gA_0=A_0\), this implies \(gc\in A_0^{\perp_b}\). Thus \(g\) preserves
	both summands \(A_0\) and \(A_0^{\perp_b}\), and its restrictions define
	elements
	\[
	g|_{A_0}\in \Aut(A_0,b|_{A_0}),
	\qquad
	g|_{A_0^{\perp_b}}\in \Aut(A_0^{\perp_b},b|_{A_0^{\perp_b}}).
	\]
	Conversely, any pair of such automorphisms defines an element of \(G\) by
	acting block-diagonally with respect to
	\[
	V=A_0\oplus A_0^{\perp_b}.
	\]
	Hence
	\[
	\operatorname{Stab}_G(A_0)
	\cong
	\Aut(A_0,b|_{A_0})
	\times
	\Aut(A_0^{\perp_b},b|_{A_0^{\perp_b}}).
	\]
\end{proof}

\begin{thm}[Nondegenerate-subspace Grassmannians]\label{thm:Nondegenerate-subspace Grassmannians}
	The orbits $\cT_k$ and $\cT_{p,q}$ are pseudo-Riemannian symmetric spaces. 
	We call them \textbf{nondegenerate-subspace Grassmannians}, as they parameterize  $k$-dimensional subspaces in $V$ on which  $b$ is nondegenerate. Explicitly, this construction produces the following seven series (here $0 \le p \le r, ~  0 \le q \le s,~  p+q = k, ~ r+s=n$):
	\[
	\frac{\OO(r,s)}{\OO(p,q)\times \OO(r-p,s-q)},\;
	\frac{\UU(r,s)}{\UU(p,q)\times \UU(r-p,s-q)},\;
	\frac{\SpG(r,s)}{\SpG(p,q)\times \SpG(r-p,s-q)},
	\]
	\[
	\frac{\SpG(2m,\mathbb{R})}{\SpG(2k,\mathbb{R})\times \SpG(2m-2k,\mathbb{R})},\;
	\frac{\SpG(2m,\mathbb{C})}{\SpG(2k,\mathbb{C})\times \SpG(2m-2k,\mathbb{C})},
	\]
	\[
	\frac{\OO(n,\mathbb{C})}{\OO(k,\mathbb{C})\times \OO(n-k,\mathbb{C})},\;
	\frac{\Ostar(2n)}{\Ostar(2k)\times \Ostar(2n-2k)}.
	\]
\end{thm}
\begin{proof}
	The homogeneous descriptions follow from Lemma \ref{lem:nondegenerate-normal-form}.  Each quotient is symmetric because the stabilizer is the fixed point group, up to connected components, of an involution of $G$. In a standard $(V_1, V_2)$ ( $\dim_{\F} V_1 = k$) basis, this involution is 
	\[
	g\longmapsto SgS^{-1},
	\qquad S|_{V_1}=1,\quad S|_{V_2}=-1,
	\]
	with the evident interpretation in each classical group.
\end{proof}

\subsection{Type IV.3: Complementary-pair Grassmannians}\label{sec:complementary-pair}
A similar construction and structure hold when $b=0$. Below, we briefly present the results for $b=0$ without proof.
Start with $V \cong \F^n$ and decompose $V$ as $V=V_1\oplus V_2$. Define
\[
s:V\to V,
\qquad s(v_1+v_2)=v_1-v_2,\qquad v_{i}\in V_{i}.
\]
Denote
\[
L=\Graph(s)=\{(v,sv):v\in V\}\subset V\oplus V, 
\]
\[
L^- =\Graph(-s)=\{(v,-sv):v\in V\} \subset V\oplus V.
\]
Then $V\oplus V=L\oplus L^-$.  On $E=V\oplus V$, the reflection along $L$ is given by
\[
R_L:E\to E,
\qquad R_L(x,y)=(sy,sx).
\]
Set $G=\GL(n, \F)$ and $\olG=\Gr_{\F}(n,E)$. Then $R_L$ induces an involution on $\olG$.

The fixed-point locus of $R_L$ is
\[
\olG^{R_L}\cong  \bigsqcup_{k=0}^{n} \Gr_{\F}(k,V) \times \Gr_{\F}(n-k,V),
\]

For each $k$, its graph-like part is 
\[
\cX_k:=\olG^{R_L}_k \cap \olGo \cong \{ (A,C) \in \Gr_{\F}(k,V) \times \Gr_{\F}(n-k,V):A \oplus C = V \} =\cT_k.
\]
\begin{thm}[ Complementary-pair Grassmannians]\label{thm: Complementary-pair Grassmannians}
	The manifolds $\cT_k$  are pseudo-Riemannian symmetric spaces. 
	We call them \textbf{complementary-pair Grassmannians}, as they parameterize  ordered pairs of transverse $k$-dimensional and $(n-k)$-dimensional subspaces in $V$. Explicitly, this construction produces the following three series:
	\[
	\frac{\GL(n,\mathbb{R})}{\GL(k,\mathbb{R})\times \GL(n-k,\mathbb{R})},\;
	\frac{\GL(n,\mathbb{C})}{\GL(k,\mathbb{C})\times \GL(n-k,\mathbb{C})},\;
	\frac{\GL(n,\mathbb{H})}{\GL(k,\mathbb{H})\times \GL(n-k,\mathbb{H})}.
	\]
\end{thm}

\bigskip
\section{Concluding remarks}\label{sec:rmk}
In this paper, we first realize classical groups as graph-like loci of their corresponding Grassmannian compactifications. Furthermore, we establish geometric realizations of other classical pseudo-Riemannian symmetric spaces as moduli spaces parameterizing classical forms, compatible scalar structures, and subspace configurations. Partial cases of these geometric descriptions have been documented in existing literature; for instance, the symplectic Grassmannian was investigated in \cite{LL}.

The fixed-locus construction separates into two cases:

(1) The quotient space \(G/H\) coincides with the graph-like locus of \(\olG^R\) associated with an involution \(R\) on  \(\olG\). In this case, the variety \(\olG^R\) provides a natural Grassmannian compactification of \(G/H\).

(2) The quotient space \(G/H\) arises as a signature stratum of the graph-like locus of \(\olG^R\) for some involution \(R\)  on \(\olG\). All symmetric spaces of this type are enumerated as follows:
\[\frac{\GL(n,\mathbb{R})}{\OO(p,q)},\;
\frac{\GL(n,\mathbb{C})}{\UU(p,q)},\;
\frac{\GL(n,\mathbb{H})}{\SpG(p,q)},\]
\[\frac{\SpG(2n, \mathbb{R})}{\UU(p,q)},\;
\frac{\SpG(2n, \mathbb{C})}{\SpG(p,q)},\;
\frac{\OO(n,\mathbb{C})}{\OO(p,q)},\;
\frac{\Ostar(2n)}{\UU(p,q)},\]
\[\frac{\OO(r,s)}{\OO(p,q) \times \OO(r-p,s-q)},\;
\frac{\UU(r,s)}{\UU(p,q) \times \UU(r-p,s-q)},\;
\frac{\SpG(r,s)}{\SpG(p,q) \times \SpG(r-p,s-q)}.\]
In the special case where \(q=0\), the space \(X_{p,q}\) reduces to a classical non-compact Riemannian symmetric space, which admits an interpretation as a space-like Grassmannian, as established in \cite{CL}. Notably, this family exhausts all classical non-compact Riemannian symmetric spaces \cite{Helgason}.

In addition, we explicitly compute the reflective submanifolds \(\olG^R\) of the Grassmannian   \(\olG\) via the aforementioned constructive arguments. Combined with the results in \cite{L1,L2}, our constructions recover many classical reflective submanifolds of Grassmannians. This work motivates a natural open problem: to obtain a complete geometric description of reflective submanifolds in classical Grassmannians.

\bigskip

School of Mathematics, East China University of Science and Technology, Shanghai 200237, China

E-mail address: yxchen76@ecust.edu.cn

\bigskip

The Institute of Mathematical Sciences and Department of Mathematics, The
Chinese University of Hong Kong, Shatin, N.T., Hong Kong

E-mail address: leung@ims.cuhk.edu.hk
\end{document}